\documentclass[12 pt]{article}
\usepackage{onetime}
\title{A New Family of Somos-like Recurrences}
\author{Paul Heideman\\\small University of Wisconsin\\\small Madison, WI 53705\\
\small \texttt{ppheideman@wisc.edu}
 \and Emilie Hogan\\\small Department of Mathematics\\
 \small Rutgers, The State University of New Jersey\\\small Piscataway, NJ 08854-8019\\
  \small \texttt{eahogan@math.rutgers.edu} }

\usepackage{amssymb}
\usepackage{amsthm}
\usepackage{amsfonts}
\usepackage{amsmath}

\newtheorem{theorem}{Theorem}
\newtheorem{lemma}{Lemma}
\newtheorem{conjecture}{Conjecture}

\newcommand{\Z}{\mathbb{Z}}

\newcommand{\R}{\mathbb{R}}

\begin{document}
\maketitle
\bibliographystyle{plain}

\begin{abstract}
We exhibit a three parameter infinite family of quadratic recurrence
relations inspired by the well known Somos sequences. For one
infinite subfamily we prove that the recurrence generates an
infinite sequence of integers by showing that the same sequence is
generated by a linear recurrence (with suitable initial conditions).
We also give conjectured relations among the three parameters so
that the quadratic recurrences generate sequences of integers.
\end{abstract}

\section{Introduction}\label{intro}
It is always the case that quadratic recurrence relations generate
sequences of rational numbers, given that the first finitely many
terms are $1$. That there exist quadratic recurrences which generate
only integers is very surprising. Michael Somos, through
investigations of elliptic theta functions, discovered one such
quadratic recurrence which began a new wave of research into this
phenomenon \cite{DGale}.  After Somos' discovery \cite{Somos} that
the recurrence
$$S(n) = \bigg(S(n-1)S(n-5) + S(n-2)S(n-4) + S(n-3)^{2}\bigg) / S(n-6)$$
\noindent generated integers when given initial conditions $S(i)=1$
for $0 \leq i \leq 5$, a surge of research began on related
sequences. Most importantly for the goals of our problem was the
sequence discovered by Dana Scott defined by the recurrence
\begin{equation}\label{DScott}
D(n) = \bigg(D(n-1)D(n-3)+D(n-2)\bigg)/D(n-4)
\end{equation}
with initial conditions $D(i)=1$ for $0 \leq i \leq 3$. This
recurrence differs from Somos' original recurrence as it mixes
quadratic and linear terms. However, it also defines a sequence
consisting only of integers.

The above two sequences possess an even stronger property than
integrality. If the initial conditions are given as $S(m)=x_{m}$
(resp. $D(m)=x_m$) for $0 \leq m \leq 5$ (resp. $0 \leq m \leq 3$),
these recurrences produce Laurent polynomials. The set of Laurent
polynomials are defined in \cite{Mathworld} as
$$ \R \left[t, t^{-1}\right] = \Big\{\ \sum a_{i} t^i ~|~ i \in \Z ,
a_{i} \neq 0 ~\textrm{for finitely many}~ i \in \Z \Big\}\ $$

\noindent We will define Laurent polynomials in more than one
variable in the same manner as above. That is, the set of Laurent
polynomials in the variables $\{x_i\}_{i=1}^n$ is defined to be
\begin{align}
\R \left[\{x_i\}_{i=1}^n,\{x_i^{-1}\}_{i=1}^n\right] &= \label{LaurentDef}\\
\Big\{\sum a_{i} \bigg(\prod_{j=1}^n x_j^{n_j}\bigg) &|~ i,n_j \in
\Z , a_{i} \neq 0 ~\textrm{for finitely many}~ i \in \Z \Big\}\
\notag
\end{align}
This Laurent property is discussed in Fomin and Zelevinsky's article
\cite{FZ}, and has been used to extract combinatorial information
from sequences.

The work done by Dana Scott was important to our discovery since
it gave us the idea to probe for sequences defined by a mixture of
quadratic and linear terms, rather than just the Somos-like
quadratic recurrences. Additionally, work done by Reid Barton
\cite{Barton} on a similar sequence resulted in a combinatorial
interpretation, which led us to believe that a similar
interpretation of our sequences exists.

\section{Statement of Problem}\label{statement}
While looking at the family of quadratic recurrences given by
\begin{equation} \label{eqnlittlek}
a(n)a(n-k)=a(n-1)a(n-k+1)+a(n-(k-1)/2)+a(n-(k+1)/2)
\end{equation} with initial conditions $a(m) = 1$ for $m \leq k$
we noticed that they seemed to produce a sequence of integers for
any odd value of $k$ we chose. By substituting $k=2K+1$ we can
transform the recurrence so it is valid for any positive integer
value of $K$. This defines the following recurrence relation
\begin{equation}\label{maineqn}
a(n) a(n-(2K+1)) = a(n-1) a(n-2K) + a(n-K) + a(n-K-1)
\end{equation}
now with initial conditions $a(m) = 1$ for $m \leq 2K+1$.

For instance, when $K=1$ the recurrence becomes
\begin{displaymath}
a(n)a(n-3)=a(n-1)a(n-2)+a(n-1)+a(n-2)
\end{displaymath}
generating the sequence $\{a(n)\} = \{1,1,1,3,7,31,85...\}$.

We have proven the integrality assertion by showing that the
sequence produced by the quadratic recurrence also satisfies a
linear recurrence with integer coefficients and integer initial
conditions. It has also been shown that  when the initial conditions
are given as formal variables the recurrences produce Laurent
polynomials (see \cite{Hogan}). Our eventual goal is to find a
combinatorial interpretation (e.g. counting perfect matchings of
certain graphs).

\section{Proof of Integrality}\label{proof}
In order to prove that the quadratic recurrence produces integers we
show that the sequence also satisfies a linear recurrence with
integer coefficients.

\begin{theorem}\label{thm}
The recurrence
\begin{equation*}
a(n) a(n-(2K+1)) = a(n-1) a(n-2K) + a(n-K) + a(n-K-1)
\end{equation*}
with initial conditions $a(m) = 1$ for $m \leq 2K+1$ generates an
infinite sequence of integers.
\end{theorem}

\begin{lemma}\label{lemma1}
The initial $6K+1$ terms of the sequence given by the quadratic
recurrence relation (\ref{maineqn}) satisfy the following
conditions:
\begin{displaymath}
\begin{array}{rclc}
a(i) &=& 1 & 1 \leq i \leq 2K+1 \\
  a(2K+i) &=& -1+2i & 2 \leq i \leq K+1 \\
  a(3K+i) &=& 1+2K-2i+2i^2 & 2 \leq i \leq K+1 \\
  a(4K+i) &=& -3-8K+2i-2K^2+12Ki+2i^2+4K^2i & 2 \leq i \leq K+1 \\
  a(5K+i) &=& 3+10K-10i+16K^2-16Ki+8i^2 & \\
  & & +4K^3-4K^2i+16Ki^2+4K^2i^2 & 2 \leq i \leq K+1
\end{array}
\end{displaymath}
\end{lemma}

\begin{proof}
Notice that $a(i) = 1$ for $1 \leq i \leq 2K+1$ by definition. Each
of the other relations is proved independently by induction. All
initial relations were originally conjectured and proved via a
computer program written by Doron Zeilberger \cite{Zeilberger}. We
will now show the proof for $a(2K+i)=-1+2i$ with $2 \leq i \leq
K+1$.

Base case $i=2$: We need to verify that $a(2K+2) = 3$. From the
quadratic definition of the sequence we have
\begin{align*}
a(2K+2)a(1) &= a(2K+1)a(2)+a(K+2)+a(K+1)\\
a(2K+2) \cdot 1 &= 1 \cdot 1 + 1 + 1\\
\Rightarrow a(2K+2) &= 3
\end{align*}
Now we assume, as the inductive hypothesis, that $a(2K+i) = -1+2i$.
We need to show that $a(2K+(i+1)) = -1+2(i+1) = 2i+1$ for $2 \leq i
\leq K$. We know that
\begin{displaymath}
a(2K+(i+1))a(i) = a(2K+i)a(i+1)+a(K+(i+1))+a(K+i)
\end{displaymath}
Since $2 \leq i \leq K$, we have
\begin{eqnarray*}
3 \leq & i+1 & \leq K+1\\
K+3 \leq & K+(i+1) & \leq 2K+1\\
K+2 \leq & K+i & \leq 2K
\end{eqnarray*}
Therefore $a(i+1)=a(K+(i+1))=a(K+i)=1$.
\begin{align*}
a(2K+(i+1)) \cdot 1 &= (-1+2i) \cdot 1 + 1 + 1\\
\Rightarrow a(2K+(i+1)) &= 2i+1
\end{align*}
So by induction, $a(2K+i)=-1+2i$ for $2 \leq i \leq K+1$.

Proofs of the other initial relations follow the same series of
steps, and will therefore be left up to the reader to verify.
\end{proof}

\begin{lemma}\label{lemma2}
If the sequence $\{a(n)\}$ is given by the quadratic recurrence
(\ref{maineqn}), then it also satisfies the linear recurrence given
by
\begin{equation} \label{lin}
a(n)=\left[2K^2+8K+4\right](a(n-2K)-a(n-4K))+a(n-6K)
\end{equation}
for all $n \geq 6K+2$, where the initial $6K+1$ values are taken to
be the first $6K+1$ values of the quadratic recurrence
(\ref{maineqn}).
\end{lemma}

\begin{proof}
First note that by uniqueness, proving the converse (i.e.- that the
sequence given by the linear recurrence \eqref{lin} satisfies the
quadratic recurrence \eqref{maineqn}) is equivalent to proving that
the statement itself. Thus we prove that the linear recurrence given
by (\ref{lin}) satisfies the quadratic recurrence given by
(\ref{maineqn}) using strong induction. Define the sequence
$\{a(n)\}$ recursively for all $n \geq 6K+2$ by \eqref{lin} and let
$a(n)$ for $1 \leq n \leq 6K+1$ be defined by the initial conditions
given in Lemma \ref{lemma1}. For simplicity in notation, let the
term $2K^2+8K+4$ be called $A(K)$. The proof will use methods
analogous to those used by Hal Canary \cite{Canary} to prove
integrality of the Dana Scott recurrence \eqref{DScott}.

Let
\begin{equation} \label{phin}
\phi(n) = a(n) a(n-(2K+1)) - a(n-1) a(n-2K) - a(n-K) -
        a(n-K-1)
\end{equation}
We wish to prove by induction that $\phi(n) = 0$ for all $n \in
\Z^+$. Clearly for $1 \leq n \leq 6K+1$, $\phi(n) = 0$ since the
first $6K+1$ terms are defined to be the first terms given by
(\ref{maineqn}).

For the base case we must prove that $\phi(6K+2) = 0$. This is
nothing but algebraic calculations easily verified by a computer
algebra system such as Maple or Mathematica.
\begin{displaymath}
\phi(6K+2) = a(6K+2) a(4K+1) - a(6K+1) a(4K+2) - a(5K+2) - a(5K+1)
\end{displaymath}
where we can substitute the initial conditions for the linear
recurrence for all but $a(6K+2)$.
\begin{align*}
a(4K+1)   &= a(3K+(K+1)) = 2K^2+4K+1\\
a(6K+1)   &= a(5K+(K+1)) = 4K^4+24K^3+40K^2+16K+1\\
a(4K+2) &= 6K^2+16K+9\\
a(5K+2) &= 4K^3+24K^2+42K+15\\
a(5K+1)   &= a(4K+(K+1)) = 4K^3+6K^2+10K+1
\end{align*}
And $a(6K+2)$ must be given by the linear recurrence (\ref{lin})
\begin{align*}
a(6K+2) &= (2K^2+8K+4)(a(4K+2)-a(2K+2))+a(2)\\
         &= (2K^2+8K+4)((6K^2+16K+9)-3)+1\\
         &= 12K^4+80K^3+164K^2+112K+25
\end{align*}
Now it is a matter of plugging this into your favorite computer
algebra system and verifying that $\phi(6K+2) = 0$.

Since the base case is verified we can proceed with the induction.
We make the strong induction assumption that $\phi(m)=0$ for $m <
n$. We need to show that $\phi(n) = 0$ given the inductive
hypothesis. Now, compute $\phi(n)$ by substituting for $a(n)$,
$a(n-1)$, $a(n-K)$, and $a(n-K-1)$ from the definition of
$\{a(n)\}$.
\begin{align*}
 a(n) &=  A(K) a(n-2K) - A(K) a(n-4K)+ a(n-6K)\\
 a(n-1) &= A(K) a(n-2K-1) - A(K) a(n-4K-1) + a(n-6K-1)\\
 a(n-K) &= A(K)  a(n-3K) - A(K)  a(n-5K) + a(n-7K)\\
 a(n-K-1) &= A(K)  a(n-3K-1) - A(K)  a(n-5K-1) + a(n-7K-1)
\end{align*}

\noindent After substituting into $\phi(n)$, expand and then
simplify and we are left with
\begin{align*}
\phi(n) =  &  -A(K) a(n-4K) a(n-2K-1) + A(K) a(n-4K-1) a(n-2K)\\
& - A(K) a(n-3K) - A(K) a(n-3K-1) + a(n-2K-1) a(n-6K)\\
& - a(n-2K) a(n-6K-1) + A(K) a(n-5K) + A(K) a(n-5K-1)\\
& - a(n-7K) - a(n-7K-1)
\end{align*}

\noindent You can see that
\begin{align*}
-A(K) a(&n-4K)a(n-2K-1)+ A(K) a(n-4K-1) a(n-2K) +\\
&- A(K) a(n-3K) - A(K) a(n-3K-1) = -A(K) \phi(n-2K)
\end{align*}
 which equals 0 by the induction hypothesis. Thus
\begin{align*}
\phi(n) = {} &  a(n-2K-1) a(n-6K) - a(n-2K) a(n-6K-1) + \\
             & + A(K) a(n-5K) + A(K) a(n-5K-1)- a(n-7K) + \\
             & - a(n-7K-1)
\end{align*}
Substitute for $a(n-2K)$ and $a(n-2K-1)$ from the definition of
$\{a(n)\}$.
\begin{align*}
 a(n-2K) = {} & A(K) a(n-4K) - A(K)  a(n-6K) + a(n-8K)\\
 a(n-2K-1) = {} & A(K)  a(n-4K-1) - A(K)  a(n-6K-1) + \\
                & + a(n-8K-1)
\end{align*}

\noindent Simplify again to obtain
\begin{align*}
\phi(n) = {} & A(K) a(n-4K-1) a(n-6K) - A(K) a(n-4K) a(n-6K-1) +\\
             & - A(K) a(n-5K) - A(K) a(n-5K-1) + \\
             & + a(n-6K) a(n-8K-1) - a(n-6K-1)a(n-8K) + \\
             & - a(n-7K)-a(n-7K-1)\\
        = {} & -A(K) \phi(n-4K) + \phi(n-6K)\\
        = {} & 0
\end{align*}
Thus by induction $\phi(n)=0$ for all $n \in \Z^+$.
\end{proof}

Theorem \ref{thm} now follows directly from the two lemmas. The
first $6K+1$ terms of the sequence generated by the quadratic
recurrence must be integers since they are given by polynomials with
integer coefficients in the variables $K$ and $i$. The rest of the
terms in the sequence satisfy a linear recurrence with integer
coefficients and integer initial conditions. Thus the terms must all
be integers.

\section{Laurentness}\label{Laurent}
Now, instead of declaring that the initial conditions of a
recurrence, $r(n)$, be all 1, we will set $r(i)$ to be the variable
$x_i$ for the first sufficiently many terms. If all terms in a
sequence are Laurent polynomials in the initial variables $\{x_i\}$
then we say that the sequence (or the recurrence producing the
sequence) has the Laurent property. In Section \ref{intro}, we gave
examples of some sequences which were conjectured to be integers,
and that have been shown to possess the Laurent property. Notice
that a sequence having the Laurent property immediately proves
integrality of that sequence. Since definition \eqref{LaurentDef} is
equivalent to defining the set of Laurent polynomials to be the
subset of rational functions in which the denominator is a monomial,
setting all variables equal to 1 in a Laurent polynomial clearly
produces an integer. For this reason, and because the Laurent
property has been helpful in producing a combinatorial proof of
integrality in some cases, many people prefer proofs of Laurentness
versus other proofs of integrality. In \cite{FZ}, Fomin and
Zelevinsky give easily verifiable sufficient conditions for a
recurrence to possess the Laurent property. Using these conditions
one can prove that the sequence given by \eqref{maineqn} with
initial conditions $a(i)=x_i$ for $1 \leq i \leq 2K+1$ has the
Laurent property. For the proof see \cite{Hogan}. We do not include
the proof here because the machinery used, namely Fomin and
Zelevinsky's conditions via cluster algebras, are more advanced than
our proof via linear recurrence. The technique of showing that a
sequence generated by a quadratic recurrence also satisfies a linear
recurrence, as in the proof from Section \ref{proof}, has the
potential to be applied in instances where the recurrence in
question does not satisfy Fomin and Zelevinsky's conditions but
still possesses the Laurent property, or even when the recurrence
satisfies integrality without Laurentness.

\section{Generalization of this Family of Quadratic
Recurrences}\label{Gen}

There is evidence to suggest that an even broader family of
quadratic recurrence relations produces integer sequences. Let the
sequence $\{b(n)\}_{n=1}^\infty$ be generated by the following
recurrence.
\begin{equation} \label{generaleqn}
b(n)b(n-k)=b(n-i)b(n-k+i)+b(n-j)+b(n-k+j)
\end{equation}

\noindent with the conditions that $i< k-i < k$, $j < k-j < k$, and
$b(l)=1$ for all $l < k$. With certain constraints on $k, i$, and
$j$, the sequence $\{b(n)\}_{n=1}^\infty$ is conjectured to be an
integer sequence.

\begin{conjecture}\label{GenFormConj}
Consider the general form of the quadratic recurrence
(\ref{generaleqn}) with initial terms $b(l)=1$ for $0 \leq l \leq
k$.
\begin{itemize}
\item In the case where $k$ is even:
\begin{itemize}
\item If $i$ is odd, then $j=\frac{k}{2}$ defines a recurrence that
generates only integers.

\item If $i$ is even, then $j=\frac{i}{2}$, $j=\frac{k}{2}$, and
$j=\frac{k-i}{2}$ define recurrences that generate only integers.\\
\end{itemize}
\item In the case where $k$ is odd:
\begin{itemize}
\item If $i$ is odd then $j=\frac{k-i}{2}$ defines a recurrence that
generates only integers.

\item If $i$ is even then $j=\frac{i}{2}$ defines a recurrence that
generates only integers.
\end{itemize}
\end{itemize}

\noindent Furthermore, all other values of $j$ do not define a
recurrence that gives integers exclusively.
\end{conjecture}

For example, the recurrence (\ref{maineqn}) is a special case of the
generalization (\ref{generaleqn}) where $k$ is odd, $i=1$, and
$j=\frac{k-1}{2}$ in agreement with Conjecture \ref{GenFormConj}.

\section{Lifting the recurrences to 2-dimensional space}\label{Lift}

As mentioned in Section \ref{Laurent}, these recurrences satisfy
even stronger conditions than integrality. However, the Laurent
polynomials which are produced will have increasingly large
coefficients. For the purposes of extracting combinatorics, it is
preferable to modify these recurrences so that the sequences are
Laurent polynomials with all coefficients of $1$ (which we will call
\emph{faithful} polynomials) so
that each term counts some object (e.g. a perfect matching).\\
\begin{conjecture}\label{conjecture03} For all conjectured
families of recurrences in section 4 except for the case where $k$
is even and $j=\frac{k}{2}$, the two dimensional recurrence
\begin{align*}
T(n,~k)T(n-k,~k)= {} & T(n-i,~k+2)T(n-k+i,~k-2)+\\
                     & + T(n-j,~k+1)+T(n-k+j,~k-1)
\end{align*}
with initial terms $T(i,j)=x_{i,j}$ for all $i < 0$ will generate
faithful Laurent polynomials in $x_{i,j}$.\end{conjecture} Note that
when $k$ is even and $j=\frac{k}{2}$, we still obtain a Laurent
polynomial with this method, but by combining linear terms one can
see immediately that it is not faithful.

Though it is not entirely clear what this lifting process will give
us, it has been a means for finding combinatorics in similar
sequences such as the Reid Barton and Dana Scott recurrences
\cite{Barton}.

\section{Conclusion}\label{conclude}
The integrality proof given here only sheds light on a small
fraction of this new family of quadratic recurrences. It is our hope
to eventually prove that all of the sequences generated by the
quadratic recurrences in Conjecture \ref{GenFormConj} satisfy linear
recurrences, and thus are integer sequences. The fact that the
Laurent property is also satisfied in this special case (and
possibly in the other cases of Conjecture \ref{GenFormConj}) may
lead to the discovery of some family of combinatorial objects which
are counted by this family of integer sequences.

\section{Acknowledgements}\label{acknowledge}
We would like to thank James Propp for suggesting possible paths to
take in this research and for running the Spatial Systems Laboratory
at the University of Wisconsin-Madison where all of our research was
conducted. Thanks also goes to the rest of the Spatial Systems
Laboratory members for their helpful contributions along the way. We
would also like to acknowledge and extend thanks to the NSF's
Research Experiences for Undergraduates program and the NSA for
funding our research. Finally we would like to thank Doron
Zeilberger of Rutgers University for helping us make the final leap
in the formulation and proof of Lemma \ref{lemma1}.

\bibliography{NewFamSomosRecBib}

\begin{thebibliography}{1}

\bibitem{Barton}
Reid Barton.
\newblock Personal communication, 2004.

\bibitem{Canary}
Hal Canary.
\newblock Personal communication, 2004.

\bibitem{FZ}
Sergey Fomin and Andrew Zelevinsky.
\newblock {T}he {L}aurent phenomenon.
\newblock {\em Adv. in Appl. Math}, 28(2):119--144, 2002.

\bibitem{DGale}
D.~Gale.
\newblock {M}athematical entertainments: The strange and surprising saga of the
  {S}omos sequences.
\newblock {\em Mathematical Intelligencer}, 13:40--42, 1991.

\bibitem{Hogan}
Emilie Hogan.
\newblock An infinite family of quadratic recurrences satisfying the {L}aurent
  phenomenon.
\newblock In preparation, 2008.

\bibitem{Somos}
Michael Somos.
\newblock Somos 6-sequence.
\newblock http://grail.csuohio.edu/$\sim$somos/somos6.html.

\bibitem{Mathworld}
Eric Weisstein.
\newblock {L}aurent polynomial.
\newblock http://www.mathworld.com/LaurentPolynomial.html.
\newblock From MathWorld--A Wolfram Web Resource.

\bibitem{Zeilberger}
Doron Zeilberger.
\newblock Maple program ``{E}milie{P}aul(k,n)".
\newblock \phantom{http://www.math.rutgers.edu/}
  http://www.math.rutgers.edu/$\sim$zeilberg/tokhniot/EMILIE.

\end{thebibliography}

\end{document}